\documentclass[12pt]{amsart}
\usepackage{amsmath,amsfonts,amsbsy,amsgen,amscd,mathrsfs,amssymb,amsthm}
\usepackage[usenames,dvipsnames]{xcolor}
\usepackage[colorlinks=true,citecolor=red,linkcolor=blue]{hyperref}
\usepackage{geometry}
\usepackage{setspace}
\allowdisplaybreaks[4]
\numberwithin{equation}{section}
\newtheorem{theorem}{Theorem}[section]
\newtheorem{definition}[theorem]{Definition}
\newtheorem{lemma}[theorem]{Lemma}

\newtheorem{proposition}[theorem]{Proposition}

\newtheorem{problem}[theorem]{Problem}

\title[$L_p$-dual mixed Euclidean-Gauss Minkowski problem]
{The solvability of the $L_p$-dual mixed Euclidean-Gauss Minkowski problem}

\author{Na Fu}
\address{Na Fu \newline \indent School of Science, Lanzhou University of Technology, Lanzhou, 730050, China}
\curraddr{}
\email{funa930323@163.com}
\thanks{}

\author{Jian-Ping Sun}
\address{Jian-Ping Sun \newline \indent School of Science, Lanzhou University of Technology, Lanzhou, 730050, China}
\curraddr{}
\email{jpsun@lut.edu.cn}
\thanks{}

\author{Bin Chen}
\address{Bin Chen \newline \indent School of Science, Lanzhou University of Technology, Lanzhou, 730050, China}
\curraddr{}
\email{chenb121223@163.com}

\author{Ya-Hong Zhao}
\address{Ya-Hong Zhao \newline \indent School of Science, Lanzhou University of Technology, Lanzhou, 730050, China}
\curraddr{}
\email{348705932@qq.com}
\thanks{Research is supported by the National Natural Science Foundation of China (Grant No. 11661049 and 12601195)}
\thanks{Corresponding author: Jian-Ping Sun}



\begin{document}
\begin{abstract}
  This paper investigates the $L_p$-dual mixed Euclidean-Gauss Minkowski problem associated with the $L_p$-dual mixed Euclidean-Gaussian curvature measure. The solvability of this problem is equivalent to that of a class of Monge-Amp\`{e}re type equations on the unit sphere. Under smooth assumptions, We establish the existence of solutions for $p<1$, $q<0$, and $q<p$ and uniqueness of solutions for $p\geq q$ to the aforementioned Monge-Amp\`{e}re type equations, respectively. Furthermore, we obtain a complete solution  concerning the existence part of this problem when $p<1$ and $q<0$.
\end{abstract}

\maketitle

\vskip 20pt
\section{Introduction }
Herein, we conduct our analysis within the $n$-dimensional Euclidean space $\mathbb{R}^n$. A convex body in $\mathbb{R}^n$ is conventionally defined as a compact convex subset possessing a  non-empty interior. Let $\mathbb{S}^{n-1}$ stand for the unit sphere in $\mathbb{R}^n$, and let $\mathcal{K}_o^n$ denote the family of all convex bodies in $\mathbb{R}^n$ that contain the origin in their respective interiors. For arbitrary vectors $x,y\in\mathbb{R}^n$,  $\langle x,y\rangle$ signifies the canonical Euclidean inner product, and $|x|=\sqrt{\langle x,x\rangle}$ denotes the associated Euclidean norm. We adopt $\mathcal{H}^k$ torepresent the $k$-dimensional Hausdorff measure and $\partial K$ to denote the boundary of a convex body $K$.

As a central topic in geometric analysis, the Minkowski problem has attracted sustained attention since its introduction by Minkowski \cite{M}. Over the past three decades, this subject has witnessed substantial progress, including the $L_p$, Orlicz, and their dual Minkowski-type problems, as well as Minkowski-type problems concerning solutions to boundary value problems. A great many influential works have emerged; see, for instance, \cite{Lu,HaL,GaH1,GaH2,HLYZ,LYZ,CL,LL,CF,ZZ1,J,Co,ZX,CD} and the
references therein. In recent years, the Gaussian Minkowski problem, an important extension of the classical Minkowski problem, was originally proposed and systematically investigated by Huang et al. \cite{HXZ}. Building upon this foundational work, Liu, Feng and other scholars further developed a profound general framework, yielding its $L_p$ variant and the corresponding dual counterpart; see \cite{BL,CS,CW,FHX,FLX,FS,LY,L,W} and the
references therein.

Motivated by the aforementioned literature, we investigate the $L_p$-dual Minkowski problem for mixed Euclidean-Gauss weighted dual quermassintegrals. To this end, we first define the definition of such weighted dual quermassintegrals.
For $q\in\mathbb{R}$ and $K\in\mathcal{K}_o^n$, define mixed Euclidean-Gauss weighted dual quermassintegrals, written as $\tilde{V}_{mix,q}(K)$, by
\begin{align}\label{1.1}
\tilde{V}_{mix,q}(K)=\int_{\mathbb{R}^n\backslash K}\exp\Big(-\tfrac12\sum_{i=1}^s\langle z,e_i\rangle^2\Big)|z|^{q-n}dz,
\end{align}
where $\langle\cdot,\cdot\rangle$ denotes the standard inner product.
In particular, setting $s=n$ and $s=0$ in \eqref{1.1} recovers the Gauss dual quermassintegral (see \cite{FLX}) and the dual quermassintegral (see \cite{S}), respectively.

From the definition of $\tilde{V}_{mix,q}$, a new geometric measure arises from the following variational formula.
For $K\in\mathcal{K}_o^n$ and $q\in\mathbb{R}$, let $f: \mathbb{S}^{n-1}\rightarrow\mathbb{R}$ be continuous. Then
$$\lim_{t\rightarrow0}\frac{\tilde{V}_{mix,q}([h_t])
-\tilde{V}_{mix,q}(K)}{t}=-\frac{1}{p}
\int_{\mathbb{S}^{n-1}}f(u)^pd\tilde{C}_{p,\tilde{V}_{mix,q}}
(K,u),$$
where, for $p\neq0$, $h_t(u)=(h_K(u)^p+tf(u)^p)^{\frac{1}{p}}$ denotes the $L_p$-combination (see \cite{Fi}), and the associated Wulff shape $[h_t]$ is defined as
$$[h_t]=\bigcap_{u\in\mathbb{S}^{n-1}}\{x\in\mathbb{R}^n: \langle x,u\rangle\leq h_t(u).$$
We call this new geometric measure $\tilde{C}_{p,\tilde{V}_{\mathrm{mix},q}}(K,\cdot)$ the $L_p$-dual mixed Euclidean-Gaussian curvature measure. As will be demonstrated later, it admits the integral representation
\begin{align}\label{1.2}
\tilde{C}_{p,\tilde{V}_{\mathrm{mix},q}}(K,\omega)
=\int_{\nu_K^{-1}(\omega)}\langle x,\nu_K(x)\rangle^{1-p}
\exp\Big(-\tfrac12\sum_{k=0}^s\langle x,e_k\rangle^2\Big)
|x|^{q-n}d\mathcal{H}^{n-1}(x),
\end{align}
for every Borel measurable set $\omega\subset\mathbb{S}^{n-1}$.

Accordingly, the $L_p$-dual mixed Euclidean-Gauss Minkowski problem is posed as follows:
\begin{problem}\label{p1.1}
Given $p,q\in\mathbb{R}$,
find necessary and sufficient conditions on a nonzero finite Borel measure $\mu$ on $\mathbb{S}^{n-1}$ guaranteeing the existence of a convex body $K$ satisfying
\begin{align*}
\mu=\tilde{C}_{p,\tilde{V}_{\mathrm{mix},q}}(K,\cdot)?
\end{align*}
If $K$ exists, to what extent is it unique$?$
\end{problem}

For $K\in\mathcal{K}_o^n$, if $\partial K$ is smooth with positive Gauss curvature, then $\tilde{C}_{p,\tilde{V}_{\mathrm{mix},q}}(K,\cdot)$ is absolutely continuous with respect to the spherical Lebesgue measure, i.e.,
\begin{align}\label{1.3}
\nonumber\frac{d\tilde{C}_{p,\tilde{V}_{\mathrm{mix},q}}(K,u)}{du}
&=h_K^{1-p}(u)
\exp\Big(-\tfrac12\sum_{k=1}^s\langle\nabla h_K(u)+h_K(u)u,e_k\rangle^2\Big)\\
&\ \ \ \ \cdot(|\nabla h_K(u)|^2+h_K^2(u))^{\frac{q-n}{2}}
\det(\nabla^2h_K(u)+h_K(u)I),
\end{align}
for all $u\in\mathbb{S}^{n-1}$. When $\mu$ possesses a non-negative integrable density $f$ on $\mathbb{S}^{n-1}$, Problem \ref{p1.1} reduces to solving the following Monge-Amp\`{e}re type equation in view of \eqref{1.3}:
\begin{align}\label{1.4}
\nonumber &h_K^{1-p}(u)
\exp\Big(-\tfrac12\sum_{k=1}^s\langle\nabla h_K(u)+h_K(u)u,e_k\rangle^2\Big)\\
&\ \ \ \ \cdot(|\nabla h_K(u)|^2+h_K^2(u))^{\frac{q-n}{2}}
\det(\nabla^2h_K(u)+h_K(u)I)=f.
\end{align}
Here, $\nabla h_K$ and $\nabla^2 h_K$ denote, respectively, the gradient vector and the Hessian matrix of the support function $h_K$ with respect to an orthonormal frame on $\mathbb{S}^{n-1}$, and $I$ denotes the identity matrix.

This paper presents the following contributions concerning the existence and uniqueness of solutions for Problem \ref{p1.1} and equation \eqref{1.4}.

First, we employ the continuity method to establish the existence of smooth solutions to equation \eqref{1.4}.
\begin{theorem}\label{t1.1}
Let $f$ be a positive smooth function on $\mathbb{S}^{n-1}$. If $p<1$, $q<0$ and $q<p$, then equation \eqref{1.4} admits a smooth solution.
\end{theorem}

Secondly, by employing Theorem \ref{t1.1}, we provide a complete resolution for the existence part of Problem \ref{p1.1} by an approximation argument.
\begin{theorem}\label{t1.2}
Let $\mu$ be a nonzero finite Borel measure on $\mathbb{S}^{n-1}$. Suppose $p<1$ and $q<0$. Then there exists a convex body $K\in\mathcal{K}_o^n$ satisfying $\mu=\tilde{C}_{p,\tilde{V}_{\mathrm{mix},q}}(K,\cdot)$ if and only if $\mu$ is not concentrated on any closed hemisphere.
\end{theorem}

Finally, we prove the uniqueness of solutions to equation \eqref{1.4}.
\begin{theorem}\label{t1.3}
Let $f$ be a positive smooth function on $\mathbb{S}^{n-1}$.
Suppose $p,q\in\mathbb{R}$ and $p\geq q$. Then equation \eqref{1.4} admits a unique solution.
\end{theorem}

The paper is structured as follows.
Section \ref{S2} presents an overview of several key facts.
Section \ref{S3} constructs an $L_p$-variational formula for a mixed Euclidean-Gaussian weighted dual quermassintegral, thereby introducing the $L_p$-dual mixed Euclidean-Gaussian curvature measure and its associated Minkowski problem.
Section \ref{S4} studies the existence and uniqueness of solutions to Problem \ref{p1.1}

\section{Preliminaries}\label{S2}
In this section, we present an overview of several key facts, for which readers may consult two renowned works by Gardner and Schneider \cite{G,S}, as well as the papers \cite{HLYZ}.

Denote by $\mathcal{K}_e^n$ the set of all origin-symmetric convex bodies in $\mathbb{R}^n$. We will write $C(\mathbb{S}^{n-1})$ for the set of continuous functions on $\mathbb{S}^{n-1}$, and $C^+(\mathbb{S}^{n-1})$ for the set of strictly positive functions in $C(\mathbb{S}^{n-1})$.

For a convex body $K$, the support function, $h(K,\cdot): \mathbb{S}^{n-1}\rightarrow\mathbb{R}$, is defined by
$$h_K(u)=h(K,u)=\max_{u\in\mathbb{S}^{n-1}}\{u\cdot y: y\in K\}.$$
We clearly know that the support function is a continuous and homogeneous convex function with degree one. The Minkowski sum of two convex bodies $K, L$ is defined, for $t>0$, by
$$K+tL=\{x+ty: x\in K, \ y\in L\}.$$
It is easy to see that the set $K+tL$ is also a convex body, and its support function is given by the following expression
$$h(K+tL,\cdot)=h(K,\cdot)+th(L,\cdot).$$
For each $K\in\mathcal{K}_o^n$, the radial function $\rho(K,\cdot)$ is defined by
$$\rho_K(x)=\rho(K,x)=\max\{\lambda: \lambda x\in K\},\ x\in\mathbb{R}^n\backslash\{0\}.$$
Obviously, the radial function is a continuous function and homogeneous of degree $-1$. Moreover, $\partial K=\{\rho_K(v)v: v\in\mathbb{S}^{n-1}\}$.

The polar body $K^*$ of $K\in\mathcal{K}_o^n$ is a convex body in $\mathbb{R}^n$ defined by
$$K^*=\{x\in\mathbb{R}^n: \langle x,y\rangle\leq1,\ \text{for all} \ y\in K\}.$$
The relation between a convex body and its polar body can be represented via their support and radial functions:
\begin{align}\label{2.1}
h_K(x)=\frac{1}{\rho_{K^*}(x)} \ \ \text{and}\ \
\rho_K(x)=\frac{1}{h_{K^*}(x)},
\end{align}
for $x\in\mathbb{R}^n\backslash\{0\}$. It follows that $K^*\in\mathcal{K}_o^n$ and
\begin{align}\label{2.2}
(K^*)^*=K.
\end{align}

For $K\in\mathcal{K}_o^n$, let $h_K(u_{\max})=\max_{u\in\mathbb{S}^{n-1}}h_K(u)$ and $\rho_K(v_{\min})=\min_{u\in\mathbb{S}^{n-1}}\rho_K(v)$. Then the following facts hold (see \cite{CL})
\begin{align}\label{2.3}
\max_{u\in\mathbb{S}^{n-1}}h_K(u)=\max_{v\in\mathbb{S}^{n-1}}\rho_K(v),\ \ \min_{u\in\mathbb{S}^{n-1}}h_K(u)=\min_{v\in\mathbb{S}^{n-1}}\rho_K(v),
\end{align}
\begin{align}\label{2.4}
h_K(u)\geq \langle u,u_{\max}\rangle h_K(u_{\max}) \ \text{for any}\ u\in\mathbb{S}^{n-1},
\end{align}
\begin{align}\label{2.5}
\rho_K(v)\langle v,v_{\min}\rangle\leq \rho_K(v_{\min}) \ \text{for any}\ v\in\mathbb{S}^{n-1}.
\end{align}

For a sequence of convex bodies $K_i\in\mathcal{K}_o^n$, we say that $K_i\rightarrow K_0$ with respect to the Hausdorff metric, provided when $i\rightarrow\infty$
$$\|h_{K_i}-h_{K_0}\|_{\infty}=\max_{u\in\mathbb{S}^{n-1}}
|h_{K_i}(u)-h_{K_0}(u)|\rightarrow0,$$
or equivalently
$$\|\rho_{K_i}-\rho_{K_0}\|_{\infty}=\max_{v\in\mathbb{S}^{n-1}}
|\rho_{K_i}(v)-\rho_{K_0}(v)|\rightarrow0.$$

For each $f\in C^+(\mathbb{S}^{n-1})$, the Wulff shape $[f]$  generated by $f$ is the convex body defined by
$$[f]=\bigcap_{u\in\mathbb{S}^{n-1}}\{x\in\mathbb{R}^n:\ \langle x,u\rangle\leq f(u),\ \text{for all }\ u\in\mathbb{S}^{n-1}\}.$$
It is apparent that $h_{[f]}\leq f$ and $[h_K]=K$ for $K\in\mathcal{K}_o^n$. The convex hull $\langle f\rangle$ generated by $f$ is defined by
$$\langle f\rangle=\text{conv}\{f(u)u: u\in\mathbb{S}^{n-1}\}
\in\mathcal{K}_o^n.$$
Obviously, if $K\in\mathcal{K}_o^n$, then $\langle\rho_K\rangle=K$.
It was proved to the fact (see \cite[Lemma 2.8]{HLYZ}) that
for $f\in C^+(\mathbb{S}^{n-1})$,
\begin{align}\label{2.6}
[f]^*=\langle \frac{1}{f}\rangle.
\end{align}

Let $K, L\in\mathcal{K}_o^n$ and $a,b\geq0$. The $L_p$-Minkowski combination, $a\cdot K+_pb\cdot L$, is defined by the Wulff shape
$$a\cdot K+_pb\cdot L=\bigcap_{u\in\mathbb{S}^{n-1}}\{x\in\mathbb{R}^n: \langle x,u\rangle\leq(ah_K(u)^p+bh_L(u)^p)^{\frac{1}{p}}\}.$$
Moreover, the $L_p$-Minkowski combination can also be defined via the support function
$$h(K+_pt\cdot L,\cdot)^p=h(K,\cdot)^p+th(L,\cdot)^p.$$
In particular, when $p=1$, the $L_p$-Minkowski combination is just the classical Minkowski sum.

The supporting hyperplane of $K\in\mathcal{K}_o^n$ for each $u\in\mathbb{S}^{n-1}$ is defined by
$$H_K(u)=\{x\in\mathbb{R}^n: \langle x,u\rangle=h_K(u)\}.$$
Let $\sigma\subset\partial K$. The spherical image of  $\sigma$ is defined by
$$\boldsymbol{\nu}_K(\sigma)=\{u\in\mathbb{S}^{n-1}: x\in H_K(u)\ \text{for some}\ x\in\sigma\}\subset\mathbb{S}^{n-1}.$$
Suppose that $\sigma_K\subset\partial K$ is the set consisting of all $x\in\partial K$ for which the set $\boldsymbol{\nu}_K(\{x\})$, often abbreviated as $\boldsymbol{\nu}_K(x)$, contains more than a single element. It is well known that $\mathcal{H}^{n-1}(\sigma_K)=0$ (see Schneider \cite[p.84]{S}). Let $\nu_K(x)$ be the unique element in $\boldsymbol{\nu}_K(x)$ for each $x\in\partial K\backslash\sigma_K$. Then we can define the function
$$\nu_K: \partial K\backslash\sigma_K\rightarrow\mathbb{S}^{n-1},$$
which is called the spherical image map (also known as the Gauss map) of $K$ and is continuous from \cite[Lemma 2.2.12]{S}. It will occasionally be convenient to abbreviate $\partial K\backslash\sigma_K$ by $\partial^\prime K$.

From integral equation \cite[(2.31)]{HLYZ}, we obtain that for each bounded integral function $f: \mathbb{S}^{n-1}\rightarrow \mathbb{R}$,
\begin{align}\label{2.7}
\int_{\partial^\prime K}f(\nu_K(x))\frac{\langle x,\nu_K(x)\rangle}{|x|^n}d\mathcal{H}^{n-1}(x)
=\int_{\mathbb{S}^{n-1}}f(\alpha_K(u))du.
\end{align}

The fact of convex body $K$ with outer normal $x\in\mathbb{R}^n$ is defined by
\begin{align}\label{2.8}
F(K,x)=\{y\in K:\ h_K(x)=\langle x,y\rangle\},
\end{align}
which lies in $\partial K$ provided $x\neq0$, and
\begin{align}\label{2.9}
F(K,x)=\partial h_K(x),
\end{align}
where $\partial h_K(x)$ is the subgradient of $h_K$. That is
$$\partial h_K(x)=\{z\in\mathbb{R}^n:\ h_K(y)\geq h_K(x)+\langle z,y-x\rangle \ \text{for each} \ y\in K\},$$
which is a non-empty compact convex set. Note that $h_K(x)$ is differentiable at $x\in \mathbb{R}^n$ if and only if $\partial h_K(x)$ consists of exactly one vector which is the gradient, denoted by $Dh_K(x)$, of $h_K$ at $x$.

The gradient of $h_K$ on $\mathbb{S}^{n-1}$ is denoted by $\nabla h_K$, provided that $h_K$ is viewed as restricted to $\mathbb{S}^{n-1}$. Since $h_K$ is differentiable at $\mathcal{H}^n$ almost all points in $\mathbb{R}^n$ and is positively homogeneous of degree one, $h_K$ is differentiable for $\mathcal{H}^{n-1}$ almost all points in $\mathbb{S}^{n-1}$. Let $h_K$ be differentiable at $u\in\mathbb{S}^{n-1}$ for a convex body $K$, where $u=\nu_K(x)$ is an outer unit normal vector at $x\in \partial K$. Then
\begin{align}\label{2.10}
x=\nu_K^{-1}(u)=Dh_K(u).
\end{align}
From this, it follows that
\begin{align}\label{2.11}
h_K(u)=h_K(\nu_K(x))=\langle x,\nu_K(x)\rangle=\langle Dh_K(u),u\rangle,
\end{align}
$$Dh_K(u)=\nabla h_K(u)+h_K(u)u,$$
and
\begin{align}\label{2.12}
|x|^2=|Dh_K(u)|^2=|\nabla h_K(u)|^2+h_K^2(u).
\end{align}

Let $K$ be a convex body. If $\partial K$ contains no segment then we say that $K$ is strictly convex. We say that $x$ is a $C^1$-smooth point, provided that $K$ has a unique tangential hyperplane at $x\in\partial K$. It is easy to see that $h_K$ is $C^1$ on $\mathbb{S}^{n-1}$ if and only if $K$ is strictly convex and $\partial K$ is $C^1$ if and only if each $x\in\partial K$ is $C^1$-smooth. For $h\in C^2(\mathbb{S}^{n-1})$, if its Hessian matrix $\{\nabla^2h+hI\}$ is positive definite on $\mathbb{S}^{n-1}$ then we call that $h$ is uniformly convex, which means that it is the support function of a convex body $K$ whose boundary is of $C^2$ and is strictly convex.

If $\partial K$ is smooth with positive Gauss curvature, then the surface area measure $S(K,\cdot)$ of $K$ is absolutely continuous with respect to spherical Lebesgue measure and its density is
\begin{align}\label{2.13}
\frac{dS(K,u)}{du}=\det(\nabla^2h_K(u)+h_K(u)I),
\end{align}
for any $u\in\mathbb{S}^{n-1}$.

It is well-known that the determinant of Jacobian for the radial Gauss mapping $\alpha_K$ of $K\in\mathcal{K}_o^n$ is given by
\begin{align}\label{2.14}
|\text{Jac}\alpha_K|=\frac{\rho_K^n}{h_K(\alpha_K)
\det(\nabla^2h_K(\alpha_K)+h_K(\alpha_K)I)}\  \ \text{on} \ \mathbb{S}^{n-1}.
\end{align}

According to the definition of surface area measure, it follows that for each $g\in C(\mathbb{S}^{n-1})$ and a convex body $K$
\begin{align}\label{2.15}
\int_{\partial^\prime K}g(\nu_K(x))d\mathcal{H}^{n-1}(x)=\int_{\mathbb{S}^{n-1}}g(u)dS(K,u).
\end{align}

The following is to recall the notions and facts with respect to the Monge-Amp\`{e}re measure from the survey by Trudinger-Wang \cite{TW}. For a convex function $\phi$ defined in an open convex set $\Omega$ of $\mathbb{R}^n$, we use $D\phi$ and $D^2\phi$ to denote its gradient and its Hessian, respectively. Let $\theta\in\Omega$ be a Borel subset. Define
$$N_\phi(\theta)=\bigcup_{x\in\theta}\partial\phi(x).$$
Then $\mu_\phi(\theta):=\mathcal{H}^n(N_\phi(\theta))$ is called the Monge-Amp\`{e}re measure. If the function $\phi$ is $C^2$ smooth, then the subgradient $\partial\phi$ coincides with the gradient $D\phi$. This has
\begin{align}\label{2.16}
\mu_\phi(\theta)=\mathcal{H}^n\big(\bigcup_{x\in\theta}D\phi(x)\big)
=\int_\theta\det(D^2\phi(x))d\mathcal{H}^n(x).
\end{align}
It is easy to see that the surface area measure $S(K,\cdot)$ of a convex body K in $\mathbb{R}^n$ is a Monge-Amp\`{e}re measure with $h_K$ restricted to $\mathbb{S}^{n-1}$ because it meets
$$\mu_{h_K}(\omega)=\mathcal{H}^{n-1}
\Big(\bigcup_{u\in\omega}\partial h_K(u)\Big)
=\mathcal{H}^{n-1}\Big(\bigcup_{u\in\omega}F(K,u)\Big)
=S(K,\omega),$$
for any Borel set $\omega\subset\mathbb{S}^{n-1}$.

We note that a convex function $\phi$ is the solution of a Monge-Amp\`{e}re equation in the sense of measure (or in the Aleksandrov sense), provided that it solves the corresponding integral formula for $\mu_\phi$ rather than the original formula for $\det(D^2\phi)$.

\section{The $L_p$-dual mixed Euclidean-Gauss curvature measure}\label{S3}
In this section, we introduce the $L_p$-dual mixed Euclidean-Gaussian curvature measure by constructing the $L_p$-variational formula for mixed Euclidean-Gauss weighted dual quermassintegral, and formulate the associated $L_p$-dual mixed Euclidean-Gauss Minkowski problem. Before proceeding, we require the following auxiliary result, whose concise proof can be found in \cite{L}. For the sake of completeness, we supply its detailed proof here.
\begin{lemma}\label{l3.1}
For $p\neq0$ and $K\in\mathcal{K}_o^n$,
let $f: \mathbb{S}^{n-1}\rightarrow\mathbb{R}$ be continuous, and let $\varepsilon>0$ be sufficiently small. Define $h_t\in C^+(\omega)$, for $t\in(-\varepsilon,\varepsilon)$ and $u\in\mathbb{S}^{n-1}$, by
$$h_t(u)=(h_K(u)^p+tf(u)^p)^{\frac{1}{p}}.$$
Then, for almost all $v\in\mathbb{S}^{n-1}$ with respect to the spherical Lebesgue measure,
$$\lim_{t\rightarrow0}\frac{\rho_{[h_t]}(v)
-\rho_K(v)}{t}=\frac{f(\alpha_K(v))^p}
{ph_K(\alpha_K(v))^p}\rho_K(v).$$
Moreover, there exists $M>0$ and $\varepsilon_0>0$ such that
$$|\rho_{[h_t]}(v)-\rho_K(v)|\leq M_0|t|,\ t\in(-\varepsilon_0,\varepsilon_0).$$
\end{lemma}

\begin{proof}
From \cite[Lemma 4.1]{HLYZ}, we have, for $u\in\mathbb{S}^{n-1}$ and $t\in(-\varepsilon,\varepsilon)$
\begin{align*}
&\frac{d}{dt}(\log h_{\langle\varrho_t\rangle}(u))\Big|_{t=0}
\lim_{t\rightarrow0}\frac{h_{\langle\varrho_t\rangle}(u)-
h_{\langle\varrho_0\rangle}(u)}{t}\\
&=\frac{1}{h_{\langle\varrho_t\rangle}(u)}\Big|_{t=0}
\lim_{t\rightarrow0}\frac{h_{\langle\varrho_t\rangle}(u)-
h_{\langle\varrho_0\rangle}(u)}{t}\\
&=g(\alpha^\ast_{\langle\varrho_0\rangle}(u)),
\end{align*}
i.e.,
\begin{align}\label{3.1}
\lim_{t\rightarrow0}\frac{h_{\langle\varrho_t\rangle}(u)-
h_{\langle\varrho_0\rangle}(u)}{t}
=h_{\langle\varrho_0\rangle}(u)
g(\alpha^\ast_{\langle\varrho_0\rangle}(u)),
\end{align}
where $g: \mathbb{S}^{n-1}\rightarrow \mathbb{R}$ is continuous and $\varrho_t: \mathbb{S}^{n-1}\rightarrow(0,\infty)$ is defined by
$$\log\varrho_t(u)=\log\varrho_0(u)+tg(u)+o(t,u).$$

By (\ref{2.1}) and (\ref{2.6}), one has
\begin{align}\label{3.2}
\nonumber\lim_{t\rightarrow0}\frac{\rho_{[h_t]}(u)
-\rho_K(u)}{t}&=\lim_{t\rightarrow0}
\frac{h_{[h_t]^*}^{-1}(u)
-h_{K^*}^{-1}(u)}{t}\\
\nonumber&=-\lim_{t\rightarrow0}\frac{h_{
\langle\frac{1}{h_t}\rangle}(u)
-h_{\langle\frac{1}{h_K}\rangle}(u)}
{th_{\langle\frac{1}{h_t}\rangle}(u)
h_{\langle\frac{1}{h_K}\rangle}(u)}\\
\nonumber&=-\frac{1}{h_{
\langle\frac{1}{h_K}\rangle}(u)^2}
\lim_{t\rightarrow0}\frac{h_{\langle\frac{1}{h_t}\rangle}(u)
-h_{\langle\frac{1}{h_K}\rangle}(u)}{t}\\
&=-\rho_K^2(u)\lim_{t\rightarrow0}
\frac{h_{\langle\frac{1}{h_t}\rangle}(u)
-h_{\langle\frac{1}{h_0}\rangle}(u)}{t}.
\end{align}

Let $\varrho_t=\frac{1}{h_t}$. Then
$$\log h_t=\log h_K-tg-o(t).$$
Since
$$h_t(u)=(h_K(u)^p+tf(u)^p)^{\frac{1}{p}},\ u\in\mathbb{S}^{n-1},$$
one has
$$\log h_t=\log h_K+\frac{tf^p}{ph_K^p}+o(t).$$
Thus
\begin{align*}
g=-\frac{f^p}{ph_K^p}.
\end{align*}
By (\ref{2.2}), (\ref{3.1}), (\ref{3.2}) and \cite[Lemma 2.6]{HLYZ}, we have
\begin{align*}
\lim_{t\rightarrow0}\frac{\rho_{[h_t]}(v)
-\rho_K(v)}{t}&=\frac{1}{p}\rho_K^2(v)
h_{\langle\frac{1}{h_0}\rangle}(v)
g(\alpha^\ast_{\langle\varrho_0\rangle}(v))\\
&=\frac{f(\alpha_K(v))^p}{ph_K(\alpha_K(v))^p}\rho_K(v).
\end{align*}

Since $[h_0]=K\in\mathcal{K}_o^n$ and $[h_t]\rightarrow K$ as $t\rightarrow0$ by Aleksandrov's convergence theorem for Wulff shapes (see \cite{S}), there exist $m_1, m_2\in(0,\infty)$ and $\varepsilon_0>0$ such that
$$0<m_1<\rho_{[h_t]}<m_2<\infty \ \ on \ \ \mathbb{S}^{n-1},$$
for $t\in(-\varepsilon_0,\varepsilon_0)$. Thus
\begin{align}\label{3.3}
0<\frac{\rho_{[h_t]}}{\rho_{K}}<m_3 \ \ on \ \  \mathbb{S}^{n-1},
\end{align}
for some $m_3>1$. Observe that $s-1\geq\log s$ for $s\in(0,1)$ while $s-1\leq\log s\leq m_3\log s$ for $s\in[1,m_3]$. Thus, when $s\in(0,m_3)$,
$$|s-1|\leq m_3|\log s|.$$
From (\ref{3.3}), we have
$$\bigg|\frac{\rho_{[h_t]}}{\rho_{K}}-1\bigg|\leq m_3\bigg|\log\frac{\rho_{[h_t]}}{\rho_{K}}\bigg|.$$
This, together with (\ref{3.3}) and fact that $|\log\rho_{[h_t]}-\log\rho_{K}|\leq M|t|$ where $M>0$ (by \cite[Lemma 4.1]{HLYZ}), implies that
$$|\rho_{[h_t]}-\rho_{K}|\leq m_2\rho_{K}|\log\rho_{[h_t]}-\log\rho_{K}|\leq m_2m_3M|t|=M_0|t|$$
on $\mathbb{S}^{n-1}$ for $t\in(-\varepsilon_0,\varepsilon_0)$. The proof is completed.
\end{proof}

The following variational formula induces the $L_p$-dual mixed Euclidean-Gauss curvature measure.

\begin{theorem}\label{t3.2}
Let $p\neq0$ and $K\in\mathcal{K}_o^n$. Suppose $f\in C(\mathbb{S}^{n-1})$, and define $h_t(u)=(h_K(u)^p+tf(u)^p)^{\frac{1}{p}}$, $u\in\mathbb{S}^{n-1}$.
Then
$$\lim_{t\rightarrow0}\frac{\tilde{V}_{mix,q}([h_t])
-\tilde{V}_{mix,q}(K)}{t}=-\frac{1}{p}
\int_{\mathbb{S}^{n-1}}f(v)^p
d\tilde{C}_{p,\tilde{V}_{mix,q}}(K,u).$$
\end{theorem}

\begin{proof}
From the definitions of $\tilde{V}_{\mathrm{mix},q}$ and $h_t$, we write $\tilde{V}_{\mathrm{mix},q}([h_t])$ in polar coordinates:
$$\tilde{V}_{mix,q}([h_t])=\int_{\mathbb{S}^{n-1}}
\int^\infty_{\rho_{[h_t]}}\exp\Big(-\tfrac12 r^2
\sum_{k=1}^s\langle u,e_k\rangle^2\Big)r^{q-1}drdu.$$
Define the integral function
$$\mathcal{F}(a)=\int^\infty_a \exp\Big(-\tfrac12r^2
\sum_{k=1}^s\langle u,e_k\rangle^2\Big)r^{q-1}dr, \ \ a\in[0,\infty).$$
By the mean-value theorem, for each fixed $u$, there exists $\xi$ lying between $\rho_{[h_t]}(u)$ and $\rho_K(u)$ such that
\begin{align}\label{3.4}
\nonumber|\mathcal{F}(\rho_{[h_t]})-\mathcal{F}(\rho_{K})|
&=|\mathcal{F}^\prime(\xi)||\rho_{[h_t]}-\rho_{K}|\\
&\leq \exp\Big(-\tfrac12\xi^2\sum_{k=1}^s\langle u,e_k\rangle^2\Big)\xi^{q-1}M_0|t|\leq\lambda(u)|t|,
\end{align}
where
$$\lambda(u)=\max_{\xi\in[\lambda_1,\lambda_2]}
\exp\Big(-\tfrac12\xi^2\sum_{k=1}^s\langle u,e_k\rangle^2\Big)\xi^{q-1}M_0,$$
and the constant $M_0$ is provided by Lemma \ref{l3.1}.

Combining Lemma \ref{l3.1}, estimate \eqref{3.4}, and identity \eqref{2.7}, we apply the dominated convergence theorem to interchange the limit and integration:
\begin{align*}
&\lim_{t\rightarrow0}\frac{\tilde{V}_{mix,q}([h_t])-
\tilde{V}_{mix,q}(K)}{t}
=\lim_{t\rightarrow0}\int_{\mathbb{S}^{n-1}}
\frac{\mathcal{F}(\rho_{[h_t]})-\mathcal{F}(\rho_{K})}{t}du\\
&=-\frac{1}{p}\int_{\mathbb{S}^{n-1}}\exp\Big(-\tfrac12
\sum_{k=1}^s\langle\rho_K(u)u,e_k\rangle^2\Big)
\rho_K^{q}(u)\frac{f(\alpha_K(u))^p}{h_K(\alpha_K(u))^p}du\\
&=-\frac{1}{p}\int_{\mathbb{S}^{n-1}}f(\alpha_K(u))^p
\exp\Big(-\tfrac12\sum_{k=1}^s\langle\rho_K(u)u,e_k\rangle^2\Big)
\rho_K^{q-n}(u)\frac{\rho_K^{n}(u)}
{h_K(\alpha_K(u))^p}du\\
&=-\frac{1}{p}\int_{\partial^\prime K}f(\nu_K(x))^p
\exp\Big(-\tfrac12
\sum_{k=1}^s\langle x,e_k\rangle^2\Big)
|x|^{q-n}\langle x,\nu_K(x)\rangle^{1-p}d\mathcal{H}^{n-1}(x)\\
&=-\frac{1}{p}\int_{\mathbb{S}^{n-1}}f(u)^p
d\widetilde{C}_{p,\tilde{V}_{mix,q}}(K,u).
\end{align*}
This completes the proof.
\end{proof}

This variational formula in Theorem \ref{t3.2} motivates the definition of the $L_p$-dual mixed Euclidean-Gaussian curvature measure, which we formalize below.
\begin{definition}\label{d3.3}
For $K\in\mathcal{K}_o^n$ and $p,q\in\mathbb{R}$, define the $L_p$-dual mixed Euclidean-Gauss curvature measure of $K$ by: for Borel set $\omega\subset\mathbb{S}^{n-1}$,
$$\widetilde{C}_{p,\tilde{V}_{mix,q}}(K,\omega)
=\int_{\nu_K^{-1}(\omega)}\exp\Big(-\tfrac12\sum_{k=1}^s\langle x,e_k\rangle^2\Big)(x\cdot\nu_K(x))^{1-p}|x|^{q-n}
d\mathcal{H}^{n-1}(x).$$
\end{definition}

The following result gives the weak convergence for the $L_p$-dual mixed Euclidean-Gauss curvature measure.

\begin{proposition}\label{p3.4}
If $K_i\in\mathcal{K}_o^n$ converges to $K_0\in\mathcal{K}_o^n$ in Hausdorff metric, then
$\widetilde{C}_{p,\tilde{V}_{mix,q}}(K_i,\cdot)$ converges to $\widetilde{C}_{p,\tilde{V}_{mix,q}}(K_0,\cdot)$ weakly.
\end{proposition}

\begin{proof}
Let $f: \omega\rightarrow\mathbb{R}$ be a continuous function. Suppose that $K_i\rightarrow K_0$ with respect to Hausdorff metric. Then
$\rho_{K_i}\rightarrow\rho_{K_0}$ uniformly on $\mathbb{S}^{n-1}$, and $\alpha_{K_i}\rightarrow\alpha_{K_0}$ almost everywhere on $\mathbb{S}^{n-1}$ with respect to spherical Lebesgue measure (by \cite[Lemma 2.2]{HLYZ}).
Moreover, $h_{K_i}(\alpha_{K_i},\cdot)\rightarrow
h_{K_0}(\alpha_{K_0},\cdot)$ uniformly on $\mathbb{S}^{n-1}$.
Since $K_0\in\mathcal{K}_o^n$, there exists a constant $C>0$ such that $C^{-1}B\subset K_i,K_0\subset CB$ for
sufficiently large $i$.
Thus
\begin{align*}
&\exp\Big(-\tfrac12\sum_{k=1}^s\langle \rho_{K_i}(u)u,e_k\rangle^2\Big)\rho_{K_i}^q(u)
\frac{f(\alpha_{K_i}(u))^p}{h_{K_i}(\alpha_{K_i}(u)^p}\\
&\rightarrow \exp\Big(-\tfrac12\sum_{k=1}^s\langle \rho_{K_0}(u)u,e_k\rangle^2\Big)\rho_{K_0}^q(u)
\frac{f(\alpha_{K_0}(u))^p}{h_{K_0}(\alpha_{K_0}(u)^p}
\end{align*}
uniformly on $\mathbb{S}^{n-1}$.

By Definition \ref{d3.3} and Lebesgue's dominated convergence theorem, one has
\begin{align*}
&\int_{\mathbb{S}^{n-1}}f(v)^p
d\widetilde{C}_{p,\tilde{V}_{mix,q}}(K_i,v)\\
&=\int_{\partial^\prime {K_i}}f(\nu_{K_i}(x))^p\exp\Big(-\tfrac12\sum_{k=1}^s\langle x,e_k\rangle^2\Big)(x\cdot\nu_{K_i}(x))^{1-p}
|x|^{q-n}d\mathcal{H}^{n-1}(x)\\
&=\int_{\mathbb{S}^{n-1}}f(\alpha_{K_i}(u))^p
\exp\Big(-\tfrac12\sum_{k=1}^s\langle \rho_{K_i}(u)u,e_k\rangle^2\Big)
\frac{\rho_{K_i}^q(u)}{h_{K_i}(\alpha_{K_i}(u))^{p}}du\\
&\rightarrow\int_{\mathbb{S}^{n-1}}f(\alpha_{K_0}(u))^p
\exp\Big(-\tfrac12\sum_{k=1}^s\langle \rho_{K_0}(u)u,e_k\rangle^2\Big)
\frac{\rho_{K_0}^q(u)}{h_{K_0}(\alpha_{K_0}(u))^{p}}du\\
&=\int_{\partial^\prime {K_0}}f(\nu_{K_0}(x))^p\exp\Big(-\tfrac12\sum_{k=1}^s\langle x,e_k\rangle^2\Big)(x\cdot\nu_{K_0}(x))^{1-p}
|x|^{q-n}d\mathcal{H}^{n-1}(x)\\
&=\int_{\mathbb{S}^{n-1}}f(v)^p
d\widetilde{C}_{p,\tilde{V}_{mix,q}}({K_0},v).
\end{align*}
The proof is completed.
\end{proof}

\begin{proposition}\label{p3.5}
For $p\neq0$, $q\in\mathbb{R}$ and $K\in\mathcal{K}_o^n$, the $L_p$-dual mixed Euclidean-Gauss curvature measure $\widetilde{C}_{\gamma,p,q}(K,\cdot)$ is absolutely continuous with respect to the surface area measure $S(K,\cdot)$.
\end{proposition}

\begin{proof}
For any Borel set $\omega\subset\mathbb{S}^{n-1}$ and $K\in\mathcal{K}_o^n$, the surface area measure $S(K,\cdot)$ of $K$ is defined as (see \cite{S})
$$S(K,\omega)=\mathcal{H}^{n-1}(\nu_K^{-1}(\omega)).$$

Let $S(K,\omega)=0$. It is equivalent to $\mathcal{H}^{n-1}(\nu_K^{-1}(\omega))=0$.
Since we are integrating over a set of measure zero, it follows that
$$\widetilde{C}_{\gamma,p,q}(K,\omega)
=\int_{\nu_K^{-1}(\omega)}\exp\Big(-\tfrac12\sum_{k=1}^s\langle x,e_k\rangle^2\Big)(x\cdot\nu_{K}(x))^{1-p}
|x|^{q-n}d\mathcal{H}^{n-1}(x)=0.$$
This implies that
$\widetilde{C}_{\gamma,p,q}(K,\cdot)$ is absolutely continuous with respect to the surface area measure $S(K,\cdot)$.
\end{proof}

{\bf The $L_p$-dual mixed Euclidean-Gauss Minkowski problem}.
Given $p,q\in\mathbb{R}$,
find necessary and sufficient conditions on a nonzero finite Borel measure $\mu$ on $\mathbb{S}^{n-1}$ guaranteeing the existence of a convex body $K$ satisfying
\begin{align*}
\mu=\tilde{C}_{p,\tilde{V}_{\mathrm{mix},q}}(K,\cdot)?
\end{align*}
If $K$ exists, to what extent is it unique$?$

Let $K$ be strictly convex. Combining Definition \ref{d3.3} with \eqref{2.11}, \eqref{2.12} and \eqref{2.15}, for every Borel set $\omega\subset\mathbb{S}^{n-1}$:
\begin{align*}
\tilde{C}_{p,\tilde{V}_{\mathrm{mix},q}}(K,\omega)
&=\int_{\nu_K^{-1}(\omega)}\exp\Big(-\tfrac12\sum_{k=1}^s
\langle x,e_k\rangle^2\Big)\langle x,\nu_K(x)\rangle^{1-p}|x|^{\frac{q-n}{2}}d\mathcal{H}^{n-1}(x)\\
&=\int_{\omega}\exp\Big(-\tfrac12\sum_{k=1}^s
\langle Dh_K(u),e_k\rangle^2\Big)
h_K^{1-p}(|\nabla h_K|^2+h_K^2)^{\frac{q-n}{2}}dS(K,u).
\end{align*}
That is,
\begin{align}\label{3.5}
\nonumber&d\tilde{C}_{p,\tilde{V}_{\mathrm{mix},q}}(K,\cdot)\\
&=\exp\Big(-\tfrac12\sum_{k=1}^s
\langle Dh_K(u),e_k\rangle^2\Big)
h_K^{1-p}(|\nabla h_K|^2+h_K^2)^{\frac{q-n}{2}}dS(K,u).
\end{align}
In view of \eqref{2.13} and \eqref{3.5}, if $K\in\mathcal{K}_o^n$ has a smooth boundary $\partial K$ with positive Gauss curvature, then $\tilde{C}_{p,\tilde{V}_{\mathrm{mix},q}}(K,\cdot)$ is absolutely continuous with respect to the spherical Lebesgue measure. Consequently,
\begin{align}\label{3.6}
\nonumber&\frac{d\tilde{C}_{p,\tilde{V}_{\mathrm{mix},q}}(K,u)}{du}\\
&=\exp\Big(-\tfrac12\sum_{k=1}^s
\langle Dh_K,e_k\rangle^2\Big)h_K^{1-p}(|\nabla h_K|^2+h_K^2)^{\frac{q-n}{2}}\det(\nabla^2h_K+h_KI),
\end{align}
for all $u\in\mathbb{S}^{n-1}$, where $u$ corresponds to the outward unit normal to $\partial K$.

Therefore, suppose $\mu$ possesses a nonnegative integrable density $f$ on $\mathbb{S}^{n-1}$. Then the $L_p$-dual mixed Euclidean-Gauss Minkowski problem reduces to solving the following Monge-Amp\`{e}re type equation:
\begin{align*}
\exp\Big(-\tfrac12\sum_{k=1}^s
\langle Dh_K,e_k\rangle^2\Big)h_K^{1-p}(|\nabla h_K|^2+h_K^2)^{\frac{q-n}{2}}\det(\nabla^2h_K+h_KI)=f.
\end{align*}

\section{Existence and uniqueness of solutions to Problem \ref{p1.1}}\label{S4}

\subsection{Regularity of solutions}

\begin{theorem}\label{Rt1}
Let $p, q\in\mathbb{R}$ and $K\in\mathcal{K}_o^n$.  Suppose  $d\tilde{C}_{p,\tilde{V}_{mix,q}}=fd\mathcal{H}^{n-1}$ on $\mathbb{S}^{n-1}$ with $0<c_1\leq f\leq c_2$. Then the following statements hold:

(i) If $f\in C(\mathbb{S}^{n-1})$, then $h_K\in C^{1,\alpha}(\mathbb{S}^{n-1})$ for all $\alpha\in(0,1)$.

(ii) If $f\in C^{\alpha}(\mathbb{S}^{n-1})$ for some $\alpha\in(0,1)$, then $h_K\in C^{2,\alpha}(\mathbb{S}^{n-1})$.
\end{theorem}

\begin{proof}
Let $h=h_K$, and let $K\in\mathcal{K}_o^n$ satisfy $d\tilde C_{p,\tilde V_{mix,q}}=f\,d\mathcal{H}^{n-1}$ on $\mathbb{S}^{n-1}$. By \eqref{3.5},
\begin{align}\label{R1}
\nonumber dS(K,u)&=h_K^{p-1}(u)\big(|\nabla h_K|^2+h_K^2\big)^{\frac{n-q}{2}}\\
&\ \ \ \ \cdot\exp\Big(\tfrac12\sum_{k=1}^s \langle Dh_K(u),e_k\rangle^2\Big)f(u)\,d\mathcal{H}^{n-1}(u).
\end{align}

Fix $e\in\mathbb{S}^{n-1}$ and define the tangent hyperplane $H_e=e+e^\perp$. Let $\pi:e^\perp\to\mathbb{S}^{n-1}$ be the radial projection given by

$$\pi(y)=\frac{y+e}{\sqrt{1+|y|^2}}.$$
A direct computation yields
\begin{align}\label{R2}
\langle\pi(y),e\rangle=\frac{1}{\sqrt{1+|y|^2}},
\end{align}
and the Jacobian determinant of $\pi$ is
\begin{align}\label{R3}
\operatorname{Jac}\pi(y)=\frac{1}{(1+|y|^2)^{n/2}}.
\end{align}
Set
\begin{align}\label{R4}
\phi(y)=h_K(y+e)=\sqrt{1+|y|^2}h_K(\pi(y)),
\end{align}
so $\phi:e^\perp\to\mathbb{R}$ is the restriction of $h_K$ to the hyperplane $H_e$. From \eqref{2.8} and \eqref{2.9}, we have
\begin{align}\label{R5}
\partial\phi(y)=\partial h_K(y+e)\big|_{e^\perp}=F(K,y+e)\big|_{e^\perp}
=F(K,\pi(y))\big|_{e^\perp}.
\end{align}
Equivalently,
\begin{align}\label{R6}
\partial\phi(y)=\{x-\langle x,e\rangle e: x\in\partial h_K(y+e)\}\subset H_e.
\end{align}
Since $h_K$ is positively homogeneous of degree one, its gradient $Dh_K$ is homogeneous of degree zero. Writing $u=\pi(y)$, we have $Dh_K(y+e)=Dh_K(u)$. Therefore

$$
D\phi(y)=Dh_K(y+e)\big|_{e^\perp}=Dh_K(u)\big|_{e^\perp},
$$
which implies that there exists some scalar $c\in\mathbb{R}$ such that
\begin{align}\label{R7}
Dh_K(u)=D\phi(y)-c\,e.
\end{align}
By Euler's identity for positively homogeneous functions, for all $u\in\mathbb{S}^{n-1}$,
\begin{align}\label{R8}
h_K(u)=\langle Dh_K(u),u\rangle.
\end{align}
Recall
\begin{align}\label{R9}
u=\pi(y)=\frac{y+e}{\sqrt{1+|y|^2}},
\end{align}
and
\begin{align}\label{R10}
h_K(u)=\frac{\phi(y)}{\sqrt{1+|y|^2}}.
\end{align}
Substitute \eqref{R7}, \eqref{R9} and \eqref{R10} into \eqref{R8}. Using $\langle D\phi(y),e\rangle=0$ since $D\phi(y)\in e^\perp$, we solve $c=\langle D\phi(y),y\rangle-\phi(y)$. Consequently,
\begin{align}\label{R11}
Dh_K(u)=D\phi(y)-\big(\langle D\phi(y),y\rangle-\phi(y)\big)e.
\end{align}

Let $\theta\subset e^\perp$ be a Borel set. Combining the area formula with \eqref{2.16} and \eqref{R5},
\begin{align}\label{R12}
\nonumber\int_\theta\det D^2\phi(y)d\mathcal{H}^{n-1}(y)&=
\mathcal{H}^{n-1}\Big(\bigcup_{y\in\theta}D\phi(y)\Big)\\
&=\mathcal{H}^{n-1}\Big(\bigcup_{u\in\pi(\theta)}
(F(K,u)|_{e^\perp})\Big).
\end{align}
It holds that
\begin{align}\label{R13}
\mathcal{H}^{n-1}\Big(\bigcup_{u\in\pi(\theta)}(F(K,u)|_{e^\perp})\Big)
=\int_{\pi(\theta)}\langle u,e\rangle dS(K,u).
\end{align}
Combining \eqref{R1}, \eqref{R2}, \eqref{R10}, \eqref{R11} and \eqref{R3}, \eqref{R12} and \eqref{R13}, and perform the change of variables $u=\pi(y)$:
\begin{align*}
&\int_\theta\det D^2\phi(y)\,d\mathcal{H}^{n-1}(y)\\
&=\int_{\pi(\theta)}\langle u,e\rangle\,dS(K,u)\\
&=\int_{\pi(\theta)}\langle u,e\rangle\exp\Big(\tfrac12\sum_{k=1}^s\langle Dh_K,e_k\rangle^2\Big)h_K^{p-1}\big(|\nabla h_K|^2+h_K^2\big)^{\frac{n-q}{2}}f(u)\,d\mathcal{H}^{n-1}(u)\\
&=\int_\theta\frac{\big|D\phi(y)-\big(\langle D\phi(y),y\rangle-\phi(y)\big)e\big|^{n-q}}{\phi(y)^{1-p}}\\
&\ \ \ \ \ \ \cdot\exp\Big(\tfrac12\sum_{k=1}^s\big\langle D\phi(y)-\big(\langle D\phi(y),y\rangle-\phi(y)\big)e,e_k\big\rangle^2\Big)\\
&\ \ \ \ \ \ \cdot \frac{1}{(1+|y|^2)^{\frac{n+p}{2}}}
f\bigg(\frac{y+e}{\sqrt{1+|y|^2}}\bigg)\,d\mathcal{H}^{n-1}(y).
\end{align*}
This yields that $\phi$ satisfies the following Monge-Amp\`{e}re type equation in the sense of measures on $e^\perp$:
\begin{align}\label{R14}
\nonumber&\det D^2\phi(y)\\
&=\frac{\big|D\phi(y)-\big(\langle D\phi(y),y\rangle-\phi(y)\big)e\big|^{n-q}}
{\phi(y)^{1-p}(1+|y|^2)^{\frac{n+p}{2}}}\\
\nonumber&\ \ \ \ \cdot\exp\Big(\tfrac12\sum_{k=1}^s\big\langle D\phi(y)-\big(\langle D\phi(y),y\rangle-\phi(y)\big)e,e_k\big\rangle^2\Big)
f\bigg(\frac{y+e}{\sqrt{1+|y|^2}}\bigg).
\end{align}

Let $\psi_{x_0}$ be an affine supporting function of $\phi$ at $x_0\in e^\perp$, and define $\Psi_{x_0}=\{x\in e^\perp:\phi(x)=\psi_{x_0}(x)\}$. We claim that $\Psi_{x_0}$ cannot contain a full straight line in $e^\perp$. We argue by contradiction, following the argument of Chen et al. \cite{CCL}. Assume that $\{z+t\zeta:t\in\mathbb{R}\}\subset\Psi_{x_0}$ for some $z,\zeta\in e^\perp$ with $|\zeta|=1$ and $\langle z,\zeta\rangle=0$. Without loss of generality, take $e=(0,\dots,0,-1)$. Write $z_t=z+t\zeta$. By \eqref{R6}, for each $t\in\mathbb{R}$, there exists $Y_t\in F_K(z_t,-1)$ such that
$$Y_t-\langle Y_t,e\rangle\cdot e=(p_{x_0},-1),$$
where $p_{x_0}:=\nabla\psi_{x_0}$.
Let
$$u_t=\frac{(z_t,-1)}{|(z_t,-1)|}.$$
Then
$$h_K(u_t)=\langle Y_t,u_t\rangle
=\frac{\langle p_{x_0},z_t\rangle+\langle Y_t,e\rangle}{|(z_t,-1)|}.$$
Letting $t\to\pm\infty$, we obtain $h_K(\zeta,0)+h_K(-\zeta,0)=0$. This contradicts $K\in\mathcal{K}_o^n$, which implies $h_K>0$ everywhere on $\mathbb{S}^{n-1}$.

Since $K\in\mathcal{K}_o^n$, \eqref{2.10} and \eqref{R11} ensure that the norm
$\big|D\phi(y)-\big(\langle D\phi(y),y\rangle-\phi(y)\big)e\big|$
is bounded above and below by positive constants. Combined with $0<c_1\leq f\leq c_2$, the right-hand side of \eqref{R14} is bounded between positive constants on compact subsets of $e^\perp$. By Caffarelli's theory \cite{C1}, $\phi$ is strictly convex. If $f\in C(\mathbb{S}^{n-1})$, the right-hand side of \eqref{R14} is continuous, and Caffarelli's regularity theory \cite{C1,C2} gives $\phi\in C^{1,\alpha}_{\text{loc}}(e^\perp)$ for some $\alpha\in(0,1)$. Since $\pi$ is a smooth local chart on $\mathbb{S}^{n-1}$, $h_K\in C^{1,\alpha}(\mathbb{S}^{n-1})$. This completes the proof of part (i).

It remains to verify (ii). If $f\in C^\alpha(\mathbb{S}^{n-1})$, then the right-hand side of \eqref{R14} is H\"older continuous. Applying Caffarelli's regularity result \cite{C2}, we get $\phi\in C^{2,\alpha}_{\text{loc}}(e^\perp)$, and hence $h_K\in C^{2,\alpha}(\mathbb{S}^{n-1})$.
\end{proof}

\subsection{Existence of solution}
In this subsection, we prove Theorem \ref{t1.1} by the continuity method and Theorem \ref{t1.2} via an approximation argument.
In order to achieve our purpose, we begin from the uniform estimates with respect to the upper and lower bounds of support function.

\begin{lemma}\label{El1}
Let $K$ be a convex body whose support function $h=h_K\in C^{2}(\mathbb{S}^{n-1})$ solves equation \eqref{1.4} with $p<1$, $q<0$, and $q<p$. If $c_0^{-1}<f<c_0$ for some constant $c_0>0$, then there exists some constant $c_1>0$ such that
\begin{align}\label{E1}
c_1^{-1}<h<c_1.
\end{align}
\end{lemma}

\begin{proof}
We first establish the left-hand side of \eqref{E1}. Suppose $h$ attains its minimum at $u_0\in\mathbb{S}^{n-1}$. At this point,
\begin{align}\label{E2}
\nabla h(u_0)=0, \ \ \text{and} \ \ \nabla^2 h(u_0)\geq0.
\end{align}
From equation \eqref{1.4} and \eqref{E2}, we have for $q<p$
\begin{align*}
f(u_0)&=\exp\Big(-\tfrac12\sum_{k=1}^s\langle\nabla h+hu,e_k\rangle^2\Big)h^{1-p}(|\nabla h|^2+h^2)^{\frac{q-n}{2}}\det(\nabla^2h+hI)\\
&\geq\exp\Big(-\tfrac12\sum_{k=1}^s\langle hu,e_k\rangle^2\Big)h^{q-p}\\
&\rightarrow \infty,
\end{align*}
as $h(u_0)\rightarrow0$, which contradicts the condition $f(u_0)<c_0$. Thus $h$ has a positive lower bound.

The right-hand side of \eqref{E1} can be established by employing the same methodology. Suppose $h$ attains its maximum at $u_1\in\mathbb{S}^{n-1}$. At this point,
$$\nabla h(u_1)=0, \ \ and \ \ \nabla^2 h(u_1)\leq0.$$
This, together with equation \eqref{1.4} and $q<p$, yields
$$c_0^{-1}<f(u_1)\leq\exp\Big(-\tfrac12\sum_{k=1}^s\langle hu,e_k\rangle^2\Big)h^{q-p}\rightarrow0, $$
as $h(u_1)\rightarrow\infty$, which also contradicts the condition $c_0^{-1}<f(u_1)$. Thus $h$ has an upper bound.
\end{proof}

The following lemma shows that the linearized operator of equation \eqref{1.4} is invertible at any positive admissible solution.

\begin{lemma}\label{El2}
Let $K\in\mathcal{K}_o^n$, $0<\alpha<1$, $p<1$, $q<0$, and $q<p$. Suppose the support $h=h_K\in C^{2,\alpha}(\mathbb{S}^{n-1})$ solves equation \eqref{1.4} for some positive $f\in C^\alpha(\mathbb{S}^{n-1})$. Then the linearized operator of \eqref{1.4} about $h$ is invertible.
\end{lemma}

\begin{proof}
Set $h_\varepsilon = h e^{\varepsilon\phi}$ with $\phi\in C^{2,\alpha}(\mathbb{S}^{n-1})$, and denote by $L_h(\phi)$ the linearized operator of \eqref{1.4} evaluated at $h$. By direct differentiation,
\begin{align}\label{E3}
\nonumber L_h(\phi)&=\frac{d}{d\varepsilon}\det(\nabla^2h_\varepsilon
+h_\varepsilon\delta_{ij})\big|_{\varepsilon=0}\\
\nonumber&\ \ -f\frac{d}{d\varepsilon}\Big[\exp\Big(\tfrac12
\sum_{k=1}^s\langle\nabla h_\varepsilon+h_\varepsilon u,e_k\rangle^2\Big)
h_\varepsilon^{p-1}(|\nabla h_\varepsilon|^2+h_\varepsilon^2)
^{\frac{n-q}{2}}\Big]\Big|_{\varepsilon=0}\\
\nonumber&=\omega^{ij}(h_{ij}\phi+h_i\phi_j+h_j\phi_i
+h\phi_{ij}+h\phi\delta_{ij})\\
\nonumber&\ \ -f\Big\{\exp\Big(\tfrac12
\sum_{k=1}^s\langle\nabla h+hu,e_k\rangle^2\Big)\sum_{k=1}^s\langle\nabla h+hu,e_k\rangle\\
&\ \ \ \ \cdot\Big[\langle(\nabla h+hu)\phi,e_k\rangle
+\Big\langle\frac{h\langle \nabla h,\nabla\phi\rangle}
{\nabla h+hu},e_k\Big\rangle\Big]h^{p-1}
(|\nabla h|^2+h^2)^{\frac{n-q}{2}}\\
\nonumber&\ \ +\exp\Big(\tfrac12
\sum_{k=1}^s\langle\nabla h+hu,
e_k\rangle^2\Big)(p-1)\phi h^{p-1}
(|\nabla h|^2+h^2)^{\frac{n-q}{2}}\\
\nonumber&\ \ +\exp\Big(\tfrac12
\sum_{k=1}^s\langle\nabla h+hu,
e_k\rangle^2\Big)(n-q)h^{p-1}
(|\nabla h|^2+h^2)^{\frac{n-q-2}{2}}\\
\nonumber&\ \ \ \ \cdot[\phi(|\nabla h|^2+h^2)+2h\langle\nabla h,\nabla\phi\rangle]\Big\},
\end{align}
where $\omega^{ij}$ stands for the cofactor matrix of $\{h_{ij}+h\delta_{ij}\}$.

Now take $\phi\equiv1$. By Lemma \ref{El1} and $q<p$,
\begin{align}\label{E4}
\nonumber L_h(1)&=\omega^{ij}(h_{ij}+h\delta_{ij})\\
\nonumber&\ \ -f\Big\{\exp\Big(\tfrac12
\sum_{k=1}^s\langle\nabla h+hu,e_k\rangle^2\Big)\sum_{k=1}^s\langle\nabla h+hu,e_k\rangle\\
\nonumber&\ \ \ \ \cdot\langle\nabla h+hu,e_k\rangle h^{p-1}
(|\nabla h|^2+h^2)^{\frac{n-q}{2}}\\
\nonumber&\ \ +\exp\Big(\tfrac12
\sum_{k=1}^s\langle\nabla h+hu,
e_k\rangle^2\Big)(p-1)h^{p-1}
(|\nabla h|^2+h^2)^{\frac{n-q}{2}}\\
&\ \ +\exp\Big(\tfrac12
\sum_{k=1}^s\langle\nabla h+hu,
e_k\rangle^2\Big)(n-q)h^{p-1}
(|\nabla h|^2+h^2)^{\frac{n-q-2}{2}}\Big\}\\
\nonumber&=(n-1)\det(h_{ij}+h\delta_{ij})-f\exp\Big(\tfrac12
\sum_{k=1}^s\langle\nabla h+hu,e_k\rangle^2\Big)\\
\nonumber&\ \ \ \ \cdot h^{p-1}(|\nabla h|^2+h^2)^{\frac{n-q}{2}}
\Big(\sum_{k=1}^s\langle\nabla h+hu,e_k\rangle^2+n+p-q-1\Big)\\
\nonumber&<(q-p)\det(h_{ij}+h\delta_{ij})<0.
\end{align}

Combining \eqref{E3} and \eqref{E4}, we have
\begin{align}\label{E5}
\nonumber L_h(\phi)&=\phi L_h(1)+\omega^{ij}(h_i\phi_j+h_j\phi_i)\\
&\ \ -f\exp\Big(\tfrac12
\sum_{k=1}^s\langle\nabla h+hu,
e_k\rangle^2\Big)h^{p-1}(|\nabla h|^2+h^2)^{\frac{n-q}{2}}\\
\nonumber&\ \ \ \ \cdot \Big(1+\frac{n-q}{|\nabla h|^2+h^2}\Big)\langle\nabla h,\nabla\phi\rangle.
\end{align}
Since the matrix $\{h_{ij}+h\delta_{ij}\}$ on $\mathbb{S}^{n-1}$ (an $(n-1)\times(n-1)$ positive definite matrix) has nonnegative cofactors $\omega^{ij}\ge0$.

Suppose $L_h(\phi)=0$. Let $x_0\in\mathbb{S}^{n-1}$ be a minimum point of $\phi$. At a minimum, $\nabla\phi(x_0)=0$ and the Hessian $\nabla^2\phi$ is positive semi-definite. Substituting into \eqref{E5}, all terms containing $\nabla\phi$ vanish at $x_0$, leaving
\[
0 = \phi(x_0)\, L_h(1).
\]
Since $L_h(1)<0$, we conclude $\phi(x_0)\ge0$. Hence $\phi\ge0$ everywhere on $\mathbb{S}^{n-1}$.

Similarly, let $x_1$ be a maximum point of $\phi$. There $\nabla\phi(x_1)=0$, and substituting into $L_h(\phi)=0$ gives $0=\phi(x_1)L_h(1)$. Again $L_h(1)<0$, so $\phi(x_1)\le0$, which yields $\phi\le0$ everywhere. Combining $\phi\ge0$ and $\phi\le0$, we obtain $\phi\equiv0$.

Therefore the kernel of $L_h$ is trivial. By the Fredholm alternative for elliptic operators on compact manifolds ($\mathbb{S}^{n-1}$), $L_h$ is invertible.
\end{proof}

With the help of Lemmas \ref{El1} and \ref{El2}, we are ready to prove Theorem \ref{t1.1}.
\subsection*{Proof of Theorem \ref{t1.1}}
Let $f=C_0>0$. Then, by the intermediate value theorem, there exists a constant solution $h=S$ of equation \eqref{1.4}. For $t\in[0,1]$ and $\alpha\in(0,1)$, define
$$f_t=(1-t)C_0+tf,$$
where $f$ is positive smooth function on $\mathbb{S}^{n-1}$.
We consider the following family of equation:
\begin{align}\label{E7}
\exp\Big(-\tfrac12
\sum_{k=1}^s\langle\nabla h+hu,
e_k\rangle^2\Big)h^{1-p}(|\nabla h|^2+h^2)^{\frac{q-n}{2}}
\det(\nabla^2h+hI)=f_t.
\end{align}
Since $f$ is smooth, a fortiori $0<f\in C^\alpha(\mathbb{S}^{n-1})$ for every $0<\alpha<1$. Hence there exists $C_1>0$ such that
$$C_1^{-1}<f,\ C_0<C_1.$$
and therefore
\begin{align}\label{E8}
C_1^{-1}<f_t<C_1 \ \ and \ \ f_t\in C^{\alpha}(\mathbb{S}^{n-1}).
\end{align}
By \eqref{E8} and Lemma \ref{El1}, every solution $h_t$ of \eqref{E7} satisfies the uniform two-sided bound
\begin{align}\label{E9}
C_2^{-1}<h_t<C_2
\end{align}
for some positive constant $C_2$ independent of $t$.

Define
$$I=\{t\in[0,1]:  \eqref{E7} \ \text{admits  solutions} \ h_t\in C^{2,\alpha}(\mathbb{S}^{n-1})\}.$$

It suffices to show that $I=[0,1]$; then $1\in I$ yields a $C^{2,\alpha}$ solution of \eqref{1.4}. We prove that $I$ is nonempty, open and closed in $[0,1]$.

Since the constant $S$ solves \eqref{E7} at $t=0$, we have $0\in I$, so $I\neq\emptyset$.

\emph{Openness.} Let $t_0\in I$ and $h_{t_0}\in C^{2,\alpha}$ be a solution of \eqref{E7} at $t=t_0$. The left-hand side of \eqref{E7} defines a $C^1$ map from a neighborhood of $(t_0,h_{t_0})$ in $[0,1]\times C^{2,\alpha}(\mathbb{S}^{n-1})$ into $C^\alpha(\mathbb{S}^{n-1})$, and its derivative with respect to $h$ at $(t_0,h_{t_0})$ is exactly the linearized operator $L_{h_{t_0}}$. Since $q<p$, Lemma \ref{El2} shows that $L_{h_{t_0}}$ is invertible. By the implicit function theorem, \eqref{E7} admits a solution $h_t\in C^{2,\alpha}$ for every $t$ in a neighborhood of $t_0$. Hence $I$ is open.

\emph{Closedness.} Let $(t_k)\subset I$ with $t_k\to t\in[0,1]$, and let $h_{t_k}\in C^{2,\alpha}$ be a solution of \eqref{E7} at $t=t_k$. By \eqref{E8}, \eqref{E9} and Theorem \ref{Rt1}, the family $\{h_{t_k}\}$ is uniformly bounded in $C^{2,\alpha}(\mathbb{S}^{n-1})$ and hence precompact in $C^{2,\beta}$ for every $\beta<\alpha$. Passing to a subsequence, $h_{t_k}\to h_t$ in $C^{2,\beta}$, and elliptic regularity upgrades the limit to $h_t\in C^{2,\alpha}$, which solves \eqref{E7} at $t$. Thus $t\in I$, and $I$ is closed.

Consequently $I=[0,1]$, so \eqref{1.4} possesses a solution $h\in C^{2,\alpha}(\mathbb{S}^{n-1})$ (the one obtained at $t=1$).

Finally, since $f$ is smooth and \eqref{1.4} is a uniformly elliptic Monge-Amp\`{e}re type equation on the compact manifold $\mathbb{S}^{n-1}$, the standard bootstrap regularity theory yields $h\in C^\infty(\mathbb{S}^{n-1})$. This completes the proof.
\ \ \ \ \ \ \ \ \ \ \ \ \ \ \ \ \ \ \ \ \ \ \ \ \ \ \ \ \ \ \ \ \ \ \ \ \ \ \ \ \ \ \ \ \ $\square$

\

Next, we will show the existence of weak solutions to Problem \ref{p1.1} via an approximation argument.

\subsection*{Proof of Theorem \ref{t1.2}}
We first prove sufficiency.  Assume that the nonzero finite Borel measure $\mu$ is not concentrated on any closed hemisphere of $\mathbb{S}^{n-1}$. We approximate $\mu$ by a sequence of smooth positive measures: let $\{f_i\}_{i\in\mathbb{N}}\subset C^\infty(\mathbb{S}^{n-1})$ be positive smooth functions such that $d\mu_i=f_idv\rightarrow d\mu$ weakly.
By Theorem \ref{t1.1}, for each $i\in\mathbb{N}$, there exists a smooth strictly convex body $K_i\in\mathcal{K}_o^n$ whose support function $h_{K_i}\in C^\infty(\mathbb{S}^{n-1})$ satisfies
\begin{align}\label{E10}
\exp\Big(-\tfrac12\sum_{k=1}^s\langle Dh_{K_i},e_k\rangle^2\Big)h_{K_i}^{1-p}(|\nabla h_{K_i}|^2+h_{K_i}^2)^{\frac{q-n}{2}}
\det(\nabla^2h_{K_i}+h_{K_i}I)=f_i.
\end{align}
This is equivalent to equation
\begin{align}\label{E11}
d\tilde{C}_{p,\tilde{V}_{mix,q}}(K_i,\cdot)=d\mu_i.
\end{align}
We claim that $K_i$ is uniformly bounded.

Step 1: Uniform positive lower bound for $h_{K_i}$.
Since $\mu$ is not concentrated on any closed hemisphere, there exist uniform constants $\varepsilon_0,\tau_0>0$ such that for arbitrary $u_0\in\mathbb{S}^{n-1}$.
Since the spherical Lebesgue measure is not concentrated in any closed hemisphere, there exist $\varepsilon_0>0$ and $\tau_0>0$ such that for every $u_i\in\mathbb{S}^{n-1}$,
\begin{align}\label{E12}
\int_{\{u\in\mathbb{S}^{n-1}:\langle u,u_i\rangle>\tau_0\}}du>\varepsilon_0>0.
\end{align}
Thus, there exists a constant $C_\mu$ depending on $\mu$  and $N_0>0$ such that for $i>N_0$
\begin{align}\label{E13}
\nonumber&C_\mu\geq \int_{\mathbb{S}^{n-1}}f_i(u)du\\
&=\int_{\mathbb{S}^{n-1}}\exp\Big(-\tfrac12
\sum_{k=1}^s\langle\rho_{K_i}(u)u,e_k\rangle^2\Big)
h_{K_i}^{1-p}(u)\rho_{K_i}^{\frac{q-n}{2}}(u)du.
\end{align}

Let $\rho_{K_i}(u_i)=\min_{u\in\mathbb{S}^{n-1}}\rho_{K_i}(u)$ and  $h_{K_i}(u_i)=\min_{u\in\mathbb{S}^{n-1}}h_{K_i}(u)$.
Together with \eqref{2.5}, we have that for any $u\in\mathbb{S}^{n-1}$ and $\langle u,u_i\rangle>\tau_0$,
$$\rho_{K_i}(u_i)\geq \rho_{K_i}(u)\langle u,u_i\rangle>
\tau_0 \rho_{K_i}(u),$$
i.e.,
\begin{align}\label{E14}
\rho_{K_i}(u)<\frac{\rho_{K_i}(u_i)}{\tau_0}.
\end{align}
Moreover
\begin{align}\label{E13.1}
h_{K_i}(u)>h_{K_i}(u_i).
\end{align}
Since $p<1$ and $q<0$, it follows that $1-p+\frac{q-n}{2}<0$. Combining \eqref{E12}, \eqref{E13}, \eqref{E14}, \eqref{E13.1}, and \eqref{2.3}, we obtain that
\begin{align*}
C_\mu&\geq \exp\Big(-\frac{s\rho_{K_i}^2(u_i)}{2\tau_0^2}\Big)
h_{K_i}(u_i)^{1-p}\Big(\frac{\rho_{K_i}(u_i)}
{\tau_0}\Big)^{\frac{q-n}{2}}\varepsilon_0\\
&=\exp\Big(-\frac{s\rho_{K_i}^2(u_i)}{2\tau_0^2}\Big)
\rho_{K_i}(u_i)^{1-p+\frac{q-n}{2}}
\tau_0^{\frac{n-q}{2}}\varepsilon_0
\rightarrow\infty
\end{align*}
as $\rho_{K_i}(u_i)\rightarrow0$, which is a contradiction. Therefore support function $h_{K_i}$ is uniformly bounded from below.

Step 2: Uniform upper bound for $h_{K_i}$.
Let $h_{K_i}(v_i)=\max_{v\in\mathbb{S}^{n-1}}h_{K_i}(v)$. It
follows from \eqref{2.4} that for all $v\in\mathbb{S}^{n-1}$ and $\langle v,v_i\rangle>0$,
\begin{align}\label{E15}
h_{K_i}(v)\geq h_{K_i}(v_i)\langle v,v_i\rangle.
\end{align}
Since $\mu$ is not concentrated in any closed hemisphere, there exists a constant $C_3>0$ such that for $q<0$ and sufficiently large $i$,
\begin{align}\label{E16}
\int_{\{v\in\mathbb{S}^{n-1}: \langle v,v_i\rangle>0\}}\langle v,v_i\rangle^{n-q}f_i(v)dv>C_3>0.
\end{align}
Combining \eqref{E10} and \eqref{2.3}, \eqref{E15}, and \eqref{E16}, we obtain that
\begin{align*}
&h_{K_i}^{n-p}(v_i)d\mathcal{H}^{n-1}(\mathbb{S}^{n-1})\\
&\geq \int_{\mathbb{S}^{n-1}}\rho_{K_i}^{n-p}(u)du\\
&=\int_{\mathbb{S}^{n-1}}\exp\Big(\tfrac12
\sum_{k=1}^s\langle\nabla h_{K_i}+h_{K_i}v,e_k\rangle^2\Big)
(|\nabla h_{K_i}|^2+h_{K_i}^2)^{\frac{n-q}{2}}f_i(v)dv\\
&\geq\int_{\mathbb{S}^{n-1}}(|\nabla h_{K_i}|^2+h_{K_i}^2)^{\frac{n-q}{2}}f_i(v)dv\\
&\geq\int_{\mathbb{S}^{n-1}}h_{K_i}^{n-q}(v)f_i(v)dv\\
&\geq\int_{\{v\in\mathbb{S}^{n-1}:\langle v,v_i\rangle>0\}}h_{K_i}^{n-q}(v)f_i(v)dv\\
&\geq h_{K_i}^{n-q}(v_i)\int_{\{v\in\mathbb{S}^{n-1}:\langle v,v_i\rangle>0\}}\langle v,v_i\rangle^{n-q} f_i(v)dv\\
&\geq C_3h_{K_i}^{n-q}(v_i).
\end{align*}
Namely, there exists a constant $C(n,q,C_3)$ depending on $n,\ q,\ C_3$ such that
$$
h_{K_i}(v)\leq C(n,q,C_3),
$$
for all $v\in\mathbb{S}^{n-1}$.

Combining Steps 1 and 2, Blaschke's selection theorem yields a subsequence $K_{i_j}$ such that $K_{i_j}\rightarrow K_0\in\mathcal{K}_o^n$ in Hausdorff metric. By weak convergence of the measure $\tilde{C}_{p,\tilde{V}_{mix,q}}$, it follows from \eqref{E11} that
$$\tilde{C}_{p,\tilde{V}_{mix,q}}(K_0,\cdot)=\mu.$$

It remains to prove the necessity. Since $K\in\mathcal{K}_o^n$, there exists a constant $C_4>0$ such that for every $x\in\partial K$,
$$\exp\Big(-\tfrac12
\sum_{k=1}^s\langle x,e_k\rangle^2\Big)\langle x,\nu_K(x)\rangle
^{1-p}|x|^{q-n}\geq C_4.$$
The surface area measure $S(K,\cdot)$ of a convex body $K\in\mathcal{K}_o^n$ is not concentrated in any closed hemisphere, i.e.,
\begin{align}\label{E18}
\int_{\mathbb{S}^{n-1}}\langle u,v\rangle_+dS(K,\cdot)>0,
\end{align}
for every $u\in\mathbb{S}^{n-1}$, where $\langle\cdot,\cdot\rangle_+=\max\{\langle\cdot,\cdot\rangle,0\}$. Now suppose $\tilde{C}_{p,\tilde{V}_{mix,q}}(K,\cdot)=\mu$ for some $K\in\mathcal{K}_o^n$. By Definition \ref{d3.3},
\eqref{2.15}, and \eqref{E18}, for any $u\in\mathbb{S}^{n-1}$,
\begin{align*}
&\int_{\mathbb{S}^{n-1}}\langle u,v\rangle_+d\mu(v)\\
&=\int_{\mathbb{S}^{n-1}}\langle u,v\rangle_+d\tilde{C}_{p,\tilde{V}_{mix,q}}(K,v)\\
&=\int_{\partial^\prime K}\langle u,\nu_K(x)\rangle_+
\exp\Big(-\tfrac12
\sum_{k=1}^s\langle x,e_k\rangle^2\Big)\langle x,\nu_K(x)\rangle
^{1-p}|x|^{q-n}d\mathcal{H}^{n-1}(x)\\
&\geq C_4\int_{\partial^\prime K}\langle u,\nu_K(x)\rangle_+d\mathcal{H}^{n-1}(x)\\
&=C_4\int_{\mathbb{S}^{n-1}}\langle u,v\rangle_+dS(K,v)\\
&>0.
\end{align*}
This implies that $\mu$ is not concentrated on any closed hemisphere. This completes the proof.
\ \ \ \ \ \ \ \ \ \ \ \ \ \ \ \ \ \ \ \ \ \ \ \ \ \ \ \ \ \ \  \ \ \ \ \ \ \ \ \ \  \ \ \ \ \ \ \ \ \ \ \ \ \ \ \ \ \ \ \ \ \ \ \ \ \ \ \ \ \ \ \ \ \ \ \ \ \ \ \ \ \ \ \ \ \ \ \ \ \ \ \ \ \ $\square$

\subsection{Uniqueness of solution}
This section establishes the uniqueness of solutions and completes the proof of Theorem \ref{t1.3}.

\subsection*{Proof of Theorem \ref{t1.3}}
Let $h_1$ and $h_2$ be two smooth solution of equation \eqref{1.4}.  Define
$$M=\max_{u\in\mathbb{S}^{n-1}}\frac{h_1(u)}{h_2(u)}
=\frac{h_1(\hat{u})}{h_2(\hat{u})}.$$
We prove $M \leq 1$ by contradiction. Suppose $M>1$. Then
\begin{align}\label{U2}
h_1(\hat{u})>h_2(\hat{u}).
\end{align}
At the maximizer point $\hat{u}$,
\begin{align}\label{U3}
0=\nabla M=\frac{\nabla h_1}{h_1}-\frac{\nabla h_2}{h_2}.
\end{align}
Moreover,
\begin{align}\label{U4}
0\geq\nabla^2 M=\frac{\nabla^2 h_1}{h_1}-\frac{\nabla^2 h_2}{h_2}.
\end{align}
Since $h_1$ and $h_2$ solve \eqref{1.4}, it follows from \eqref{2.11} that for any $u\in\mathbb{S}^{n-1}$
\begin{align*}
f(u)&=\exp\Big(-\tfrac12\sum_{k=1}^s\langle\nabla h_1+h_1u,e_k\rangle^2\Big)h_1^{1-p}
(|\nabla h_1|^2+h_1^2)^{\frac{q-n}{2}}
\det(\nabla^2h_1+h_1I)\\
&=\exp\Big(-\tfrac12\sum_{k=1}^s\Big\langle h_1\Big(\frac{\nabla h_1}{h_1}+u\Big),e_k\Big\rangle^2\Big)
h_1^{q-p}\Big(\Big|\frac{\nabla h_1}{h_1}\Big|^2+1
\Big)^{\frac{q-n}{2}}\det\Big(\frac{\nabla^2h_1}{h_1}+I\Big)\\
&=\exp\Big(-\tfrac12\sum_{k=1}^s\Big\langle h_2\Big(\frac{\nabla h_2}{h_2}+u\Big),e_k\Big\rangle^2\Big)
h_2^{q-p}\Big(\Big|\frac{\nabla h_2}{h_2}\Big|^2+1
\Big)^{\frac{q-n}{2}}\det\Big(\frac{\nabla^2h_2}{h_2}+I\Big).
\end{align*}
Evaluating at $\hat{u}$ and using \eqref{U3} and \eqref{U4}, we obtain
\begin{align*}
1&=\frac{\exp(-\frac{1}{2}\sum_{k=1}^s\langle h_1(\frac{\nabla h_1}{h_1}+u),e_k\rangle^2)
h_1^{q-p}\det(\frac{\nabla^2h_1}{h_1}+I)}
{\exp(-\frac{1}{2}\sum_{k=1}^s\langle h_2(\frac{\nabla h_2}{h_2}+u),e_k\rangle^2)
h_2^{q-p}\det(\frac{\nabla^2h_2}{h_2}+I)}\\
&\leq\frac{\exp(-\frac{1}{2}\sum_{k=1}^sh_1^2\langle\frac{\nabla h_1}{h_1}+u,e_k\rangle^2)}
{\exp(-\frac{1}{2}\sum_{k=1}^sh_2^2\langle\frac{\nabla h_1}{h_1}+u,e_k\rangle^2)}
\Big(\frac{h_1}{h_2}\Big)^{q-p}\\
&=\Big(\frac{\exp(\frac{h_2}{2})}{\exp(\frac{h_1}{2})}\Big)^
{\sum_{k=1}^sh_1^2\langle\frac{\nabla h_1}{h_1}+u,e_k\rangle^2}
\Big(\frac{h_1}{h_2}\Big)^{q-p}.
\end{align*}
From \eqref{U2} and the assumption $p\geq q$, we have
$$\Big(\frac{\exp(\frac{h_2^2}{2})}{\exp(\frac{h_1^2}{2})}\Big)^
{\sum_{k=1}^sh_1^2\langle\frac{\nabla h_1}{h_1}+u,e_k\rangle^2}\geq
\Big(\frac{h_1}{h_2}\Big)^{p-q}\geq1.$$
Hence
$$\exp\Big(\frac{h_2}{2}\Big)\geq\exp\Big(\frac{h_1}{2}\Big),$$
which yields $h_2(\hat{u})\geq h_1(\hat{u})$. This contradicts \eqref{U2}. Therefore,
\begin{align}\label{U5}
h_1(u)\leq h_2(u)\ \ \text{for all } u\in\mathbb{S}^{n-1}.
\end{align}

Next, define
\begin{align*}
m=\min_{u\in\mathbb{S}^{n-1}}\frac{h_1(u)}{h_2(u)}
=\frac{h_1(\check{u})}{h_2(\check{u})}.
\end{align*}
Similar to the proof above, we have
$$h_1(u)\geq h_2(u)\ \ \text{for all } u\in\mathbb{S}^{n-1}.$$
This, together with \eqref{U5}, yields
$$h_1(u)\equiv h_2(u)\ \ \text{for all } u\in\mathbb{S}^{n-1}.$$
The proof is concluded.\ \ \ \ \ \ \ \ \ \ \ \ \ \ \ \ \ \ \ \ \ \ \ \ \ \ \ \ \ \ \  \ \ \ \ \ \ \ \ \ \  \ \ \ \ \ \ \ \ \ \ \ \ \ \ \ \ \ \ \ \ \ \ \ \ \ \ \ \ \ \ \ \ \ \ \ \ \ \ \ \ \ \ \ $\square$

\subsection*{Acknowledgments}
The authors extend their heartfelt appreciation to the anonymous reviewers for their valuable feedback. This work was supported by Natural Science Foundation of China (Grant No. 11661049 and 12601195).


\vskip 1.0cm
\begin{thebibliography}{11}

{\small

\bibitem{M} H. Minkowski, {\it Volumen und Oberfl\"{a}che}, Math. Ann., {\bf 57} (1903), 447-495.

\bibitem{Lu} E. Lutwak, {\it The Brunn-Minkowski-Firey theory. I. Mixed volumes and the Minkowski problem}, J. Differential Geom., {\bf 38} (1993), 131-150.

\bibitem{HaL} C. Haberl, E. Lutwak, D. Yang and G. Y. Zhang, {\it The even Orlicz Minkowski problem}, Adv. Math., {\bf 224} (2010), 2458-2510.

\bibitem{GaH1} R. J. Gardner, D. Hug, W. Weil, S. D. Xing and D. P. Ye, {\it General volumes in the Orlicz Brunn-Minkowski theory and a related Minkowski problem I}, Calc. Var. Partial Differential Equations, {\bf 58} (2019), 1-35.

\bibitem{GaH2} R. J. Gardner, D. Hug, S. Xing and D. Ye, {\it General volumes in the Orlicz-Brunn-Minkowski theory and a related Minkowski problem II}, Calc. Var. Partial Differential Equations, {\bf 59} (2020), 1-33.

\bibitem{HLYZ} Y. Huang, E. Lutwak, D. Yang and G. Y. Zhang, {\it Geometric measures in the dual Brunn-Minkowski theory and their associated Minkowski problems}, Acta Math., {\bf 216} (2016), 325-388.

\bibitem{LYZ} E. Lutwak, D. Yang and G. Y. Zhang, {\it $L_p$ dual curvature measures}, Adv. Math.,  {\bf 329} (2018), 85-132.

\bibitem{CL} H. Chen and Q.-R. Li, {\it The $L_p$ dual Minkowski problem and related parabolic flows}, J. Funct. Anal., {\bf 281} (2021), 109139.

\bibitem{LL} Y. N. Liu and J. Lu, {\it A flow method for the dual Orlicz-Minkowski problem}, Trans. Amer. Math. Soc., {\bf 373} (2020),  5833-5853.

\bibitem{CF} A. Colesanti and M. Fimiani, {\it The Minkowski problem for the torsional rigidity}, Indiana Univ. Math. J., {\bf 59} (2010), 1013-1039.

\bibitem{ZZ1} X. Zhao and P. B. Zhao, {\it The $p$-th dual Minkowski problem for the $k$-torsional rigidity corresponding to a $k$-Hessian equation}, J. Funct. Anal., {\bf 291} (2026), 111576.

\bibitem{J} D. Jerison, {\it A Minkowski problem for electrostatic capacity}, Acta Math., {\bf 176} (1996), 1-47.

\bibitem{Co} A. Colesanti, K. Nystr\"{o}m, P. Salani, J. Xiao, D. Yang and G. Y. Zhang, {\it The Hadamard variational formula and the Minkowski problem for p-capacity}, Adv. Math., {\bf 285} (2015), 1511-1588.

\bibitem{ZX} D. Zou and G. Xiong, {\it  The $L_p$ Minkowski problem for the electrostatic $\mathfrak{p}$-capacity}, J. Differential Geom., {\bf 116} (2020), 555-596.


\bibitem{CD} B. Chen, W. D. Wang, X. Zhao and P. B. Zhao, {\it An inverse Gauss curvature flow and its application to the $p$-capacitary Orlicz-Minkowski problem}, Rev. Mat. Iberoam., {\bf 41} (2025), 2283-2308.

\bibitem{HXZ} Y. Huang, D. M. Xi and Y. M. Zhao, {\it The Minkowski problem in Gaussian probability space}, Adv. Math., {\bf 385} (2021), 107769.

\bibitem{BL} K. J. K\"{o}r\"{o}czky, E. Lutwak, D. Yang, G. Y. Zhang and Y. M. Zhao, {\it The Gauss image problem}, Commun. Pure Appl. Math., {\bf 73} (2020), 1406-1452.

\bibitem{CS} B. Chen, W. Shi and W. D. Wang, {\it An inverse Gauss curvature flow to the $L_p$-Gauss Minkowski problem}, J. Math. Anal. Appl., {\bf 541} (2025), 128656.

\bibitem{CW} L. Chen, D. Wu and N. Xiang, {\it Smooth solutions to  the Gauss image problem}, Pacific J. Math., {\bf 317} (2022), 275-295.

\bibitem{FHX} Y. B. Feng, S. Hu and L. Xu, {\it On the $L_p$ Gaussian Minkowski problem}, J. Differential Equations, {\bf 363} (2023), 350-390.

\bibitem{FLX} Y. B. Feng, Y. Li and L. Xu, {\it Existence of solutions to the Gaussian dual Minkowski problem}, J. Differential Equations, {\bf 416} (2025), 268-298.

\bibitem{FS} N. Fu and J.-P. Sun, {\it The $L_p$-Gauss dual Minkowski problem}, arXiv: 2412.13557v2 (2025).

\bibitem{LY} X. Li and D. P. Ye, {\it The Gaussian Minkowski problem for epigraphs of convex functions}, Int. Math. Res. Not. IMRN, {\bf 2026} (2026), rnag075.

\bibitem{L} J. Q. Liu, {\it The $L_p$-Gauss Minkowski problem}, Calc. Var. Partial Differential Equations, {\bf 61} (2022), Paper No. 28.

\bibitem{W} H. J. Wang, {\it Continuity of the solution to the $L_p$ Minkowski problem in Gaussian probability space},  Acta Math. Sin. (Engl. Ser.), {\bf 38} (2022), 2253-2264.

\bibitem{S} R. Schneider, {\it Convex Bodies: The Brunn-Minkowski theory}, 2nd edn, Cambridge Univ. Press, Cambridge, 2014.

\bibitem{Fi} W. J. Firey, {\it $p$-means of convex bodies}, Math. Scand., {\bf 10} (1962), 17-24.

\bibitem{G} R. J. Gardner, {\it Geometric Tomography}, second ed., Cambridge Univ. Press, New York, 2006.

\bibitem{TW} N. S. Trudinger and X.-J. Wang, {\it The Monge-Amp\`{e}re equation and its geometric applications}, Handb. Geom. Anal.,  {\bf I} (2008), 467-524.

\bibitem{CCL} H. Chen, S. Chen and Q.-R. Li, {\it Variations of a class of Monge-Amp\`{e}re type functionals and their applications}, Anal. PDE, {\bf 14} (2021), 689-716.

\bibitem{C1} L. Caffarelli, {\it A localization property of viscosity solutions to the Monge-Amp\`{e}re equation and their strict convexity}, Ann. Math., {\bf 131} (1990), 129-134.

\bibitem{C2} L. Caffarelli, {\it Interior $W^{2,p}$-estimates for solutions of the Monge-Amp\`{e}re equation}, Ann. Math., {\bf 131} (1990), 135-150.




}
\end{thebibliography}
\end{document}